\documentclass[11pt]{amsart}
\usepackage{float}
\usepackage{epsfig}
\usepackage{graphicx}
\usepackage{caption}
\usepackage{subcaption}
\usepackage{color}
\usepackage{xcolor}
\usepackage{enumerate}

\title[Emergent boundary conditions]{Diffusion laws select boundary conditions}

\author{Jaywan Chung}
\address[Jaywan Chung]{Energy Conversion Research Center, Korea Electrotechnology Research Institute, 12, Jeongiui-gil, Changwon-si, Gyeongsangnam-do, 51543, Korea}
\email{jchung@keri.re.kr}

\author{Seungmin Kang}
\address[Seungmin Kang]{Department of Mathematical Sciences, KAIST,e 291 Daehak-ro, Yuseong-gu, Daejeon, 34141, Korea}
\email{sngmn20@gmail.com}

\author{Ho-Youn Kim}
\address[Ho-Youn Kim]{Computer,  Electrical  and  Mathematical  Science  and  Engineering  Division,  King  AbdullahUniversity  of  Science  and  Technology  (KAUST),  Thuwal  23955-6900,   Saudi  Arabia}
\email{ghsl0615@gmail.com}

\author{Yong-Jung Kim}
\address[Yong-Jung Kim]{Department of Mathematical Sciences, KAIST,e 291 Daehak-ro, Yuseong-gu, Daejeon, 34141, Korea}
\email{yongkim@kaist.edu}

\def\eps{\epsilon}
\def\L{\mathcal{L}}
\def\tri{\triangle}
\def\R{\mathbb{R}}

\def\erfc{\mathrm{erfc}}
\def\ds{\displaystyle}

\newtheorem{theorem}{Theorem}[section]
\newtheorem{lemma}{Lemma}[section]

\newtheorem{definition}{Definition}[section]

\numberwithin{equation}{section}

\begin{document}
\maketitle
\begin{abstract}
The choice of boundary condition makes an essential difference in the solution structure of diffusion equations. The Dirichlet and Neumann boundary conditions and their combination have been the most used, but their legitimacy has been disputed. We show that the diffusion laws may select boundary conditions by themselves, and through this, we clarify the meaning of boundary conditions. To do that we extend the domain with a boundary into the whole space by giving a small diffusivity $\eps>0$ outside the domain. Then, we show that the boundary condition turns out to be Neumann or Dirichlet as $\eps\to0$ depending on the choice of a heterogeneous diffusion law. These boundary conditions are interpreted in terms of a microscopic-scale random walk model.
\end{abstract}

\section{Introduction and main results}

Consider a diffusion equation on the half line $\R^+$,
\begin{equation}\label{1.1}
\begin{cases}
U_t =U_{xx},& t,x>0,\\
U(0,x) = \phi(x), & x>0,
\end{cases}
\end{equation}
where $U(t,x)$ is the density of diffusing particles, $U_t$ and $U_{xx}$ are partial derivatives, and $\phi\in L^\infty(\R^+)$ is a bounded initial value. The domain $\R^+$ has a boundary at $x=0$, and the composition of the problem is complete only after an appropriate boundary condition is given at $x=0$. The importance of boundary conditions in diffusion equations has been known for a long time. The Dirichlet and Neumann boundary conditions, and their combinations, have been the most used, but their legitimacy has been disputed. The purpose of the paper is to derive a boundary condition and discuss what natural boundary conditions are in what situations. 

The strategy of the paper is to view the problem \eqref{1.1} as the limit of a diffusion problem on the whole real line $\R$ when the diffusivity is given by
\begin{equation}\label{D}
D=D(x; \epsilon):=
\begin{cases}
\epsilon \quad& \text{if $x < 0$,}\\
1 \quad& \text{if $x > 0$.}
\end{cases}
\end{equation}
If we take the limit as $\eps\to0^+$, the diffusivity disappears in the $x<0$ region, and the solution will converge to a solution of \eqref{1.1} in the region $x>0$. The boundary behavior of the limit can be considered as the natural boundary condition we are looking for. One might think the diffusing particles are trapped in the $x>0$ region since the diffusion in the $x<0$ region disappears. If so, the Neumann boundary condition will be the desired boundary condition. However, the emerging boundary condition depends on the situation. 

The diffusion equation with a nonconstant diffusivity $D$ can be written generally as
\begin{equation}\label{1.3}
\begin{cases}
u_t = (D^{1-q} (D^{q}u)_x)_x,& t>0,\ x\in\R,\\
u(0,x) =\phi(x), & x\in\R,
\end{cases}
\end{equation}
where $\phi\in L^\infty(\R)$ and $q\in\R$. If $D$ is constant, the diffusion equations in \eqref{1.3} are identical for all $q\in\R$. If $q=0$, the diffusion equation  in \eqref{1.3} is called Fick's diffusion law \cite{Fick}. If $q\ne0$, it is called a Fokker-Plank diffusion law since it contains an advection phenomenon. For example, we may split the diffusion operator into Fick's law diffusion and an advection term, i.e.,
\[
u_t = (D^{1-q} (D^{q}u)_x)_x=(Du_x+quD_x)_x.
\]
The last term $quD_x$ gives a drift effect. In particular, if $q=0.5$, it is Wereide's diffusion law \cite{Wereide} and satisfied by the probability density function of a stochastic process under the Stratonovich integral. If $q=1$, it is Chapman's diffusion law \cite{Chapman} and satisfied by the probability density functions under the It\^o integral. 

There have been lots of debates and discussions on the correct diffusion equation in heterogeneous environments. The two cases with $q=0$ and $q=1$ are mostly compared (see \cite{Landsberg,Milligen}). Kampen \cite{Kampen88} and Schnitzer \cite{Schnitzer} claimed that there is no universal form of the diffusion equation, but each system should be studied individually. Landsberg and Sattin \cite{Sattin} claimed what matters is adding an appropriate convective term, not a diffusion law. Kampen \cite{Kampen81} clarified similar debates related to stochastic differential equations for probability density functions. These long-standing controversies seem to say that there is no correct exponent $q$ over others and we need two coefficients $D$ and $q$ depending on the situation. We may split the diffusivity into two coefficients in a heterogeneous environment and write the diffusion equation as \eqref{5.2}. We use this formation to explain why different boundary conditions emerge at the interface.

The typical definition of weak solution cannot be used for the problem \eqref{1.3}. Since the diffusivity $D$ in \eqref{D} is discontinuous at $x=0$, the solution of \eqref{1.3} is not a typical weak solution. For notational convenience, we denote the domain without the interface $x=0$ as
\[
\Omega_b :=\R\setminus\{0\}=\R^- \cup \R^+.
\]
We denote the regularity with respect to the time and space variables using a sub-index as in $C^1_t$ and $C^1_x$. The solution of \eqref{1.3} in the paper is defined as follows.\footnote{Another way to define the solution of \eqref{1.3} is in a weak sense. However, since the diffusivity $D$ is discontinuous at $x=0$, the regularity should be given to the product $D^qu$ and hence it is not a usual weak solution for $0<q<1$. If $q=0$, the solution can be defined in the usual weak sense. If $q=1$, the solution can be defined in a very weak sense.}

\begin{definition}\label{def:1.1}

A function $u \in C^1_t(\R^+\times\Omega_b)\cap C(\R^+\times\Omega_b)\cap L^\infty(\R^+\times\R)$ is called a solution of \eqref{1.3} if the followings hold:\\
(i) For all $t > 0$,
\begin{eqnarray}
\label{R1}
D^{q} u \in C^1_x(\R^+\times\Omega_b) \cap C(\R^+\times\R),\\
\label{R2}
\quad D^{1-q} (D^{q}u)_x \in C^1_x(\R^+\times\Omega_b) \cap C(\R^+\times\R).
\end{eqnarray}
(ii) Eq. \eqref{1.3} is satisfied in the classical sense for all $(t,x) \in \R^+ \times \Omega_b$.\\
(iii) The initial condition is satisfied pointwise except at the interface, i.e.,
\[
\lim_{t \to 0^+} u(t,x) = \phi(x),\quad x\in\Omega_b.
\]
\end{definition}

We call $D^qu$ the \emph{potential} and $D^{1-q} (D^{q}u)_x$ the \emph{flux}. Note that the two regularity conditions, \eqref{R1} and \eqref{R2}, are not given to the solution, but to the potential and the flux. Furthermore, if $q\ne1$, $D^{q} u$ has no chance to belong to $C^1_x(\R^+\times\R)$ since $D^{1-q} (D^{q}u)_x \in C(\R^+\times\R)$ and $D^{1-q}$ is discontinuous at $x=0$. The main results of the paper are in the following theorem.

\begin{theorem}[Boundary condition and regularity]\label{thm} Let $\phi\in C(\Omega_b)\cap L^\infty(\R)$, $\phi\ge0$, $q\in\R$, and $D$ be given by \eqref{D}. Then, the solution of \eqref{1.3} exists and is unique. The solution satisfies
\begin{equation}\label{Regularity}
D^qu_t \in C(\R^+\times \R) \quad \text{and} \quad D^q(D^{1-q} (D^qu)_x)_x \in C(\R^+\times\R).
\end{equation}
The solution $u(t,x;\eps)$ converges pointwise for all $x>0$ as $\eps\to0$, and the limit, denoted by $U(t,x)$, satisfies \eqref{1.1} and a boundary condition,
\begin{equation}\label{BC}
\begin{cases}
\ds\frac{\partial U}{\partial x}(t,0^+)=0,& \text{ if }\ q<0.5,\\
\ds\ \ U(t,0^+) =0,&\text{ if }\ q>0.5.
\end{cases}
\end{equation}
Furthermore, as $\epsilon \to 0$,
\begin{equation}\label{decay}
\begin{cases}
\ds \left| \frac{\partial u}{\partial x}(t,0^+;\epsilon) \right| = O(\epsilon^{0.5 - q}) & \text{ if }\ q<0.5,\\
\ds\ \ |u(t, 0^+; \epsilon)| = O(\epsilon^{q - 0.5}) &\text{ if }\ q>0.5.
\end{cases}
\end{equation}
\end{theorem}

According to the theorem, the Neumann boundary condition emerges if $q<0.5$. Therefore, if Fick's diffusion is taken, the zero-flux condition is a natural boundary condition. On the other hand, if $q>0.5$, the Dirichlet boundary condition is the natural one. If it is the border case that $q=0.5$, there is no restriction on the boundary condition. Equations \eqref{R1} and \eqref{Regularity} point out that if $q\ne0$, the solution $u$, its time derivative, $u_t$, and the spatial derivative of the flux, $(D^{1-q}(D^qu)_x)_x$, have the same jump discontinuity at the interface $x=0$ which become continuous after multiplying by $D^q$. These regularity properties indicate that we need to multiply $D^q$ by \eqref{1.3}. Then, the pressure $p=D^qu$ satisfies
\[
p_t = D^q(D^{1-q} p_x)_x,\quad t>0,\ x\in\R.
\]
Notice that the pressure is not a conservative quantity and handling this equation is inconvenient. However, if we change the space variable with $\frac{\partial y}{\partial x}=D^{-q}(x)$, the equation becomes conservative,
\[
p_t = (D^{1-2q} p_y)_y,\quad t>0,\ y\in\R.
\]
This is the way we obtain \eqref{2.1} in the next section, and the pressure becomes a classical solution of it. We may extend the question to the multidimensional case with or without a reaction function. In fact, Hilhorst \emph{et al.} \cite{HKKK} showed that the Neumann boundary condition emerges along the boundary of a bounded domain $\Omega\subset\R^n$ if Fick's diffusion law is taken with bistable nonlinearity.

The paper consists as follows. In Section \ref{sect.proof}, Theorem \ref{thm} is proved. We find the explicit formula of the solution and then prove the claims in the theorem one by one using the explicit formula. In Section \ref{sect.examples}, two examples of explicit solutions of \eqref{1.3} are given. These explicit solutions are useful to see the dynamics of the problem. In Section \ref{sect.numerics}, we numerically test the convergence of boundary conditions as $\eps\to0$. These numerical computations show us detailed convergence order of the boundary conditions as $\eps\to0$ which are given in \eqref{decay}. In Section \ref{sect.why}, we discuss why the boundary condition in \eqref{BC} splits at $q=0.5$. We consider the relationship between the diffusion equation \eqref{1.3} and the microscopic-scale random walk model. The case $q=0.5$ corresponds to the one with a constant sojourn time. The heterogeneity in the sojourn time decides the boundary condition.

\section{Proof of Main Results}\label{sect.proof}

This section is for the proof of Theorem \ref{thm}. We first solve the problem explicitly and prove the theorem one by one which are divided into several lemmas. We reduce \eqref{1.3} to Fick's diffusion law by scaling the space variable and writing it in terms of the potential. First, introduce a new space variable
\begin{equation}\label{2.0}
y := \int_0^x D^{-q}(s) \, ds,
\end{equation}
and write the potential as
\[
p(t,y):=D^{q}(x) u(t,x), \quad  p_0(y):=D^{q}(x) \phi(x).
\]
Since $D$ is constant respectively for $x<0$ and $x>0$, we have $D(y)=D(x)$ under the relation \eqref{2.0}. Then,
\[
	p_t(t,y) = D^{q} u_t(t,x), \quad 	p_y(t,y) = D^q(D^{q} u(t,x))_x.
\]
Therefore, in terms of the new variables, the flux is written as
\[
D^{1-q}(D^{q} u)_x=D^{1-2q}(D^{q} u)_y=D^{1-2q}p_y=
\begin{cases}
\sigma p_y,& y<0,\\
\ \ p_y,&y>0,
\end{cases}
\]
where we denote $\sigma:=\eps^{1-2q}$ for simplicity. Finally, the main equation \eqref{1.3} is written as
\begin{equation}\label{2.1}
	\begin{cases}
		p_t = ( D^{1-2q} p_y)_y, \quad &t>0,\ y\in\R,\\
		p(0,y) =  p_0(y), \quad & y \in \R.
	\end{cases}
\end{equation}
The strategy to prove Theorem \ref{thm} is to solve \eqref{2.1} explicitly and prove the claims in the theorem using the explicit formula. Explicit solutions such as Gaussian, heat kernel, and Barenblatt solution have provided essential information about a general solution. The explicit solutions obtained in the paper play a similar role. 

We solve \eqref{2.1} by gluing together two solutions of boundary value problems with appropriate interface conditions. First, we divide the real line into two regions, $\{y > 0\}$ and $\{y < 0\}$. Then, in the region $\{y > 0\}$, \eqref{2.1} is written as
\begin{equation}\label{2.3}
\begin{cases}
p_t = p_{yy}, & t,y>0,\\
p(0,y) = p_0(y),\ & y>0,\\
p(t,0) = h(t),  & t>0,
\end{cases}
\end{equation}
where $h(t)$ is treated as a given boundary value for now. In the region $\{y<0\}$, \eqref{2.1} is written as
\begin{equation}\label{2.4}
\begin{cases}
p_t = \sigma p_{yy} \quad & t>0,\ y<0,\\
p(0,y) = p_0(y) \quad & y<0,\\
p(t,0) = h(t) \quad & t>0,
\end{cases}
\end{equation}
where $h(t)$ is the same boundary value. For any given boundary value, we may solve the problem and construct a solution $p$ defined on $\R$. This will satisfy all conditions in Definition \ref{def:1.1} except the continuity of the flux in \eqref{R2}. Therefore, it suffices to find an $h(t)$ that satisfies the continuity of the flux, i.e.,
\begin{equation}\label{2.5}
\sigma p_y(t,0^-) = p_y(t,0^+).
\end{equation}
In conclusion, the original problem \eqref{1.3} is written as a system \eqref{2.3}--\eqref{2.5}.

The solution of \eqref{2.3} is explicitly given by
\begin{equation}\label{pright}
	\begin{aligned}
		p(t,y) &= \int_0^\infty p_0(\xi) \Phi(t, y,\xi)\,d\xi \\
		&\quad + \int_{0^+}^t h'(\tau) \erfc\Big( \frac{y}{2\sqrt{t-\tau}} \Big)\,d\tau + h(0^+) \erfc\Big( \frac{y}{2\sqrt{t}} \Big),
	\end{aligned}\ y>0,
\end{equation}
where
\[
\Phi(t, x, \xi) := \frac{1}{2 \sqrt{\pi t}} (e^{-\frac{(x-\xi)^2}{4t}} - e^{-\frac{(x+\xi)^2}{4t}})
\]
(see \cite[Eq.(1.4.9) in p.19 and Eq.(1.4.21) in p.22]{kevorkian} for example). To use the same formula for the solution of \eqref{2.4}, we change the variables as
\[
s=\sigma t,\quad z=-y,\quad v(s,z)=p(t,y).
\]
Then, \eqref{2.4} turns into
\[
\begin{cases}
	v_s = v_{zz} \quad & s>0,\ z>0,\\
	v(0,z) = p_0(-z) \quad & z>0,\\
	v(s,0) = h(\frac{s}{\sigma}) \quad & s>0.
\end{cases}
\]
Therefore, 
\[
\begin{aligned}
	v(s,z) &= \int_0^\infty v(0,\xi) \Phi(s, z,\xi)\,d\xi \\
	&\quad + \int_{0^+}^s \frac{h'(\frac{\tau}{\sigma})}{\sigma} \erfc\Big( \frac{z}{2\sqrt{s-\tau}} \Big)\,d\tau + h(0^+) \erfc\Big( \frac{z}{2\sqrt{s}} \Big).
\end{aligned}
\]
If we return to the original variables, we obtain
\begin{equation}\label{pleft}
	\begin{aligned}
		p(t,y) &= \int_0^\infty p_0(-\xi) \Phi(\sigma t, -y,\xi)\,d\xi \\
		&\quad + \int_{0^+}^{\sigma t} \frac{h'(\frac{\tau}{\sigma})}{\sigma} \erfc\Big( \frac{-y}{2\sqrt{\sigma t-\tau}} \Big)\,d\tau + h(0^+) \erfc\Big( \frac{-y}{2\sqrt{\sigma t}} \Big),
	\end{aligned}\ y<0.
\end{equation}

The function $p(t,y)$ given by \eqref{pright} and \eqref{pleft} satisfies \eqref{2.3} and \eqref{2.4}, but not \eqref{2.5}. The first step to find $h(t)$ that makes \eqref{2.5} to be satisfied is comparing the flux at the interface from the left and the right sides. 
\begin{lemma} Let $p(t,y)$ be given by \eqref{pright} and \eqref{pleft}, and $h(t) \in C^1(\R^+)\cap C(\overline{\R^+})$. Then,
\begin{eqnarray}
\label{right derivative}
&&\frac{\partial}{\partial y} p(t,0^+) =\int_0^\infty \frac{p_0(\xi)\xi}{2\sqrt{\pi t^3}}e^{-\frac{\xi^2}{4t}} \,d\xi-\int_{0^+}^t \frac{h'(\tau)}{\sqrt{\pi(t-\tau)}}\, d\tau - \frac{h(0^+)}{\sqrt{\pi t}},\\
\label{left derivative}
&&\frac{\partial}{\partial y} p(t,0^-)=-\int_0^{\infty} \frac{p_0(-\sqrt{\sigma}\xi)\xi}{2\sqrt{\pi\sigma t^3}}e^{-\frac{\xi^2}{4t}} \,d\xi+\int_{0^+}^t \frac{h'(\tau)}{\sqrt{\pi\sigma(t-\tau)}}\, d\tau + \frac{h(0^+)}{\sqrt{\pi\sigma t}}.
\end{eqnarray}
\end{lemma}
\begin{proof}
	First, calculate $\frac{\partial}{\partial y} p(t,y)$ for $y>0$ using \eqref{pright},
	\begin{eqnarray}
		&&\begin{array}{ll}\label{prightderivative}
			\ds\frac{\partial}{\partial y} p(t,y) =& \int_0^\infty \frac{p_0(\xi)}{4\sqrt{\pi t^3}}((\xi-y)e^{-\frac{(y-\xi)^2}{4t}}+(\xi+y)e^{-\frac{(y+\xi)^2}{4t}})\, d\xi\\
			&- \int_{0^+}^t \frac{1}{\sqrt{\pi(t-\tau)}}h'(\tau) e^{-\frac{y^2}{4(t-\tau)}} \, d\tau - h(0^+)\frac{1}{\sqrt{\pi t}}e^{-\frac{y^2}{4t}}.
		\end{array}
	\end{eqnarray}
	Then, \eqref{right derivative} is derived directly by taking the limit as $y\to0^+$ in \eqref{prightderivative}.
	
	Next, calculate $\ds\frac{\partial}{\partial y} p(t,y)$ for $y<0$ using \eqref{pleft},
	\begin{eqnarray}
		&&\begin{array}{ll}\label{pleftderivative}
			\ds\frac{\partial}{\partial y} p(t,y) =&\int_0^{\infty} \frac{p_0(-\xi)}{4\sqrt{\pi \sigma^3t^3}}((\xi+y)e^{-\frac{(y+\xi)^2}{4\sigma t}}+(\xi-y)e^{-\frac{(y-\xi)^2}{4\sigma t}})\,d\xi\\
			&- \int_{0^+}^{\sigma t} \frac{h'(\frac{\tau}{\sigma})}{\sigma}\frac{1}{\sqrt{\pi(\sigma t-\tau)}} e^{-\frac{y^2}{4(\sigma t-\tau)}} \, d\tau - h(0^+)\frac{1}{\sqrt{\pi\sigma t}}e^{-\frac{y^2}{4\sigma t}}.
		\end{array}
	\end{eqnarray}
	Taking the limit as $y\to0^-$ in \eqref{pleftderivative} gives
	\[
	\frac{\partial}{\partial y}p(t,0^-) = \int_0^\infty \frac{p_0(-\xi)\xi}{2\sqrt{\pi\sigma^3 t^3}}e^{-\frac{\xi^2}{4\sigma t}} \,d\xi-\int_{0^+}^{\sigma t} \frac{h'(\frac{\tau}{\sigma})}{\sigma\sqrt{\pi(\sigma t-\tau)}}\, d\tau - \frac{h(0^+)}{\sqrt{\pi \sigma t}},
	\]
	which implies \eqref{left derivative} by taking change of variables $\xi' := \frac{\xi}{\sqrt{\sigma}}$ and $\tau' := \frac{\tau}{\sigma}$.
\end{proof}

Next, we find the $h(t)$ that makes the flux continuous, i.e., \eqref{2.5} is satisfied. We also show the $h(t)$ is uniquely decided and hence the solution of \eqref{1.3} is unique.
\begin{lemma}\label{boundedsol} Let $\sigma\in\R^+$, $p_0 \in C(\Omega_b) \cap L^\infty(\R)$, and $p_0\ge0$. The function $p(t,y)$ given by \eqref{pright} and \eqref{pleft} is the bounded solution of \eqref{2.3}--\eqref{2.5} if and only if
	\begin{equation}\label{h(t)}
		h(t) = \int_0^\infty \frac{p_0(\xi)+\sqrt{\sigma}p_0(-\sqrt{\sigma }\xi)}{1+\sqrt{\sigma}}\frac{e^{-\frac{\xi^2}{4t}}}{\sqrt{\pi t}}\,d\xi.
	\end{equation}
Furthermroe, $h(t) \in C^1(\R^+)\cap C([0,\infty))$ and $p\in C_t^1(\R^+\times\R)$.
\end{lemma}
\begin{proof}
	Suppose that \eqref{2.3}--\eqref{2.5} has a bounded solution. Putting the results \eqref{right derivative}, \eqref{left derivative} into \eqref{2.5} gives
	\begin{equation}\label{flux-continuity'''}
		\int_{0^+}^t \frac{h'(\tau)}{\sqrt{\pi (t-\tau)}}  \, d\tau + \frac{h(0^+)}{\sqrt{\pi t}}= \int_0^\infty \frac{p_0(\xi) + \sqrt{\sigma} p_0(-\sqrt{\sigma}\xi)}{(1+\sqrt{\sigma})}\frac{\xi}{2\sqrt{\pi t^3}}e^{-\frac{\xi^2}{4t}}\, d\xi.
	\end{equation}
	Take the Laplace transform of \eqref{flux-continuity'''} with respect to the variable $t$. Then we obtain a formula
	\begin{equation}\label{laplace}
		\L\{h\}(s) = \int_0^\infty \frac{p_0(\xi)+\sqrt{\sigma}p_0(-\sqrt{\sigma }\xi)}{1+\sqrt{\sigma}}\frac{e^{-\xi\sqrt{s}}}{\sqrt{s}}\,d\xi.
	\end{equation}
	Finally, by taking the inverse of Laplace transform of \eqref{laplace}, we have
	\[
	h(t) = \int_0^\infty \frac{p_0(\xi)+\sqrt{\sigma}p_0(-\sqrt{\sigma }\xi)}{1+\sqrt{\sigma}}\frac{e^{-\frac{\xi^2}{4t}}}{\sqrt{\pi t}}\,d\xi,
	\]
	which yields $h(t)\in C([0,\infty))$.
	
	To prove the lemma in the other direction, suppose that $h(t)$ is given by the formula \eqref{h(t)}. Note that
	\[|h(t)|\le \frac{\|p_0\|_{L^\infty(\R^+)}+\sqrt{\sigma}\|p_0\|_{L^\infty(\R^-)}}{1+\sqrt{\sigma}},\quad \mbox{for all $t>0$}.\]
Denote
	\[
	g(t,\xi) :=\frac{p_0(\xi)+\sqrt{\sigma}p_0(-\sqrt{\sigma }\xi)}{1+\sqrt{\sigma}}\frac{e^{-\frac{\xi^2}{4t}}}{\sqrt{\pi t}}.
	\]
	Differentiating the integrand $g$ with respect to $t>0$, we have
	\[
	\frac{\partial}{\partial t} g(t,\xi) = \frac{p_0(\xi)+\sqrt{\sigma}p_0(-\sqrt{\sigma }\xi)}{1+\sqrt{\sigma}}(\frac{\xi^2}{4t^2}-\frac{1}{2t})\frac{e^{-\frac{\xi^2}{4t}}}{\sqrt{\pi t}}.
	\]
	Since $\frac{p_0(\xi)+\sqrt{\sigma}p_0(-\sqrt{\sigma }\xi)}{1+\sqrt{\sigma}}\in L^\infty_\xi(\R^+)$ and $(\frac{\xi^2}{4t^2}-\frac{1}{2t})\frac{e^{-\frac{\xi^2}{4t}}}{\sqrt{\pi t}}\in C_t(\R^+;L^1_\xi(\R^+))$, we obtain that $\frac{\partial}{\partial t} g(t,\xi)\in C_t(\R^+;L^1_\xi(\R^+))$. Thus, we can interchange derivative and integral operators such as
	\[
	h'(t) = \frac{d}{dt}\int_0^\infty g(t,\xi)\, d\xi = \int_0^\infty \frac{\partial}{\partial t} g(t,\xi)\, d\xi,
	\]
	which yields $h(t)\in C^1([0,\infty))$. In other words, the differentiability of $p$ with respect to $t$ variable is extended to include the interface $x=0$, i.e., $p\in C_t^1(\R^+\times\R)$.
	
	Now, we can construct a solution $p$ of Problem \eqref{2.3}--\eqref{2.5} from \eqref{pright}, \eqref{pleft} and \eqref{h(t)}. We note that $p_0\in L^\infty(\R)$ and $h\in L^\infty(\R^+)$. Since $\|p\|_{L^\infty(\R^+\times\R)}\le \|p_0\|_{L^\infty(\R)} + \|h\|_{L^\infty(\R^+)}$, $p$ is bounded.
\end{proof}

We have obtained the uniqueness and the existence of the solution of \eqref{1.3}. Next, we show the rest of the theorem. The first part of \eqref{Regularity} is equivalent to the inclusion $p\in C_t^1(\R^+\times\R)$. It may be written as
\[
p_t=D^q u_t\in C(\R^+\times\R).
\]
In the next lemma, we show the second part of \eqref{Regularity}.

\begin{lemma}[Extra regularity of the flux]\label{lemma2.3} Let $p(t,y)$ be the bounded solution in Lemma \ref{boundedsol}. Then, the flux $(D^{1-2q}p_y)$ is in $C^1_y(\R^+\times\R)$, i.e.,
	\[\sigma \frac{\partial^2}{\partial y^2}p(t,0^-) = \frac{\partial^2}{\partial y^2}p(t,0^+) = h'(t).\]
\end{lemma}
\begin{proof}
	Remember that $h(t)\in C^1(\R^+)\cap C(\overline{\R^+})$ is already shown in the proof of Lemma \ref{boundedsol}. First, calculate $\frac{\partial^2}{\partial y^2}p(t,y)$ for $y>0$:
	\begin{eqnarray}
		&&\begin{array}{ll}\label{rightsecond}
			&\frac{\partial^2}{\partial y^2}p(t,y) \\
			&=\int_0^\infty \frac{p_0(\xi)}{8\sqrt{\pi t^5}}(((\xi-y)^2-1)e^{-\frac{(y-\xi)^2}{4t}}-((\xi+y)^2-1)e^{-\frac{(y+\xi)^2}{4t}}) \, d\xi\\
			&\ \ + \int_{0^+}^t \frac{y}{2\sqrt{\pi(t-\tau)^3}}h'(\tau) e^{-\frac{y^2}{4(t-\tau)}} \, d\tau + h(0^+)\frac{y}{2\sqrt{\pi t^3}}e^{-\frac{y^2}{4t}}.
		\end{array}
	\end{eqnarray}
	Taking the limit of \eqref{rightsecond} as $y\to 0^+$ gives
	\[
	\frac{\partial^2}{\partial y^2}p(t,0^+)=\lim_{y\to0^+} \int_{0^+}^t \frac{y}{2\sqrt{\pi (t-\tau)^3}}h'(\tau) e^{-\frac{y^2}{4(t-\tau)}} \, d\tau.
	\]
After taking $\eta :=\frac{y}{\sqrt{t-\tau}}$, the second-order derivative is written as
	\[
	\frac{\partial^2}{\partial y^2}p(t,0^+)=\lim_{y\to0^+} \int_{\frac{y}{\sqrt{t}}}^\infty \frac{1}{\sqrt{\pi}}h'\Big(t-\frac{y^2}{\eta^2}\Big) e^{-\frac{\eta^2}{4}}\, d\eta\\
	=h'(t).
	\]
	
	Similarly, calculate $\frac{\partial^2}{\partial y^2}p(t,y)$ for $y<0$:
	\begin{eqnarray}
		&&\begin{array}{ll}\label{leftsecond}
			&\frac{\partial^2}{\partial y^2}p(t,y) \\
			&=\int_0^\infty \frac{p_0(-\xi)}{8\sqrt{\pi \sigma^5 t^5}}(((\xi-y)^2-1)e^{-\frac{(y-\xi)^2}{4\sigma t}}-((\xi+y)^2-1)e^{-\frac{(y+\xi)^2}{4\sigma t}}) \, d\xi\\
			&\ \ + \int_{0^+}^t \frac{y}{2\sqrt{\pi\sigma^3(t-\tau)^3}}h'(\tau) e^{-\frac{y^2}{4\sigma(t-\tau)}} \, d\tau + h(0^+)\frac{y}{2\sqrt{\pi \sigma^3t^3}}e^{-\frac{y^2}{4\sigma t}}.
		\end{array}
	\end{eqnarray}
	Taking the limit of \eqref{leftsecond} as $y\to 0^-$ gives
	\[
	\frac{\partial^2}{\partial y^2}p(t,0^-)=\lim_{y\to0^+} \int_{0^+}^t \frac{y}{2\sqrt{\pi \sigma^3(t-\tau)^3}}h'(\tau) e^{-\frac{y^2}{4\sigma(t-\tau)}} \, d\tau.
	\]
After taking $\eta :=-\frac{y}{\sqrt{t-\tau}}$, the second-order derivative is written as
	\[
	\frac{\partial^2}{\partial y^2}p(t,0^-)= \lim_{y\to0^-} \int^{\frac{y}{\sqrt{t}}}_{-\infty} \frac{1}{\sqrt{\pi \sigma^3}}h'\Big(t-\frac{y^2}{\eta^2}\Big) e^{-\frac{\eta^2}{4\sigma}}\, d\eta\\
	=h'(t)/\sigma.
	\]
The comparison of the two second-order derivatives completes the proof.
\end{proof}

Lemmas \ref{boundedsol} and \ref{lemma2.3} imply that the obtained solution $p(t,y)$ satisfies \eqref{2.1} in the classical sense. That means the differentiations in the equation \eqref{2.1} is valid even at the interface $x=0$. The second part of \eqref{Regularity} is equivalent to the inclusion $(D^{1-2q}p_y)\in C^1_y(\R^+\times\R)$. It may be written as
\[
(D^{1-2q}p_y)_y=D^q(D^{1-2q}D^qp_x)_x=D^q(D^{1-q}(D^qu)_x)_x\in C(\R^+\times\R).
\]
Next, we show the boundary behavior of the solution.

\begin{lemma}[Boundary condition when $\eps\to0$] Let $\sigma\in\R^+$, $p_0 \in C(\Omega_b) \cap L^\infty(\R)$, $p_0\ge0$, and $h(t;\epsilon)$ be given by \eqref{h(t)}. Then, $p(t,y; \epsilon)$ converges to a limit $U(t,y)$ for all $y>0$ pointwise as $\eps\to0$, and
	\[
		\begin{cases}
			\ds\lim_{y\to0+}\frac{\partial U}{\partial y}(t,y)=0 & \text{ if }\ q<0.5,\\
			\ds\lim_{y\to0+}\ \ U(t,y) =0 &\text{ if }\ q>0.5.
		\end{cases}
	\]	
Furthermore, as $\epsilon \to 0$,
\begin{equation*}
\begin{cases}
\ds \left| \frac{\partial p}{\partial y}(t,0^+;\epsilon) \right| = O(\epsilon^{0.5 - q}) & \text{ if }\ q<0.5,\\
\ds \ \ |p(t, 0^+; \epsilon)| = O(\epsilon^{q - 0.5}) &\text{ if }\ q>0.5.
\end{cases}
\end{equation*}
\end{lemma}
\begin{proof}
(i) Consider the case $q<0.5$. Note that $\sigma=\eps^{1-2q}\to 0$ as $\eps\to 0$.
First, we show the existence of a limit. By \eqref{h(t)}, for a fixed $t > 0$, we have
\[
\begin{aligned}
\lim_{\epsilon \to 0} h(t; \epsilon) &= \int_0^\infty p_0(\xi) \frac{e^{-\frac{\xi^2}{4t}}}{\sqrt{\pi t}}\,d\xi, \\
\lim_{\epsilon \to 0} h'(t; \epsilon) &= \int_0^\infty p_0(\xi) \Big(\frac{\xi^2}{4t^2}-\frac{1}{2t}\Big)\frac{e^{-\frac{\xi^2}{4t}}}{\sqrt{\pi t}}.
\end{aligned}
\]
Hence $h(t; 0^+) \in C^1(\R^+)\cap C(\overline{\R^+})$. Therefore, taking the limit as $\epsilon \to 0$ in \eqref{pright}, we can see that $p(t, y; \epsilon)$ converges as $\epsilon \to 0$.

Next, we show the boundary behavior. By \eqref{right derivative} and \eqref{flux-continuity'''},
\[
\begin{aligned}
\frac{\partial p}{\partial y} (t,0^+; \epsilon) &= \int_0^\infty \frac{p_0(\xi)\xi}{2\sqrt{\pi t^3}}e^{-\frac{\xi^2}{4t}} \,d\xi-\int_{0^+}^t \frac{h'(\tau)}{\sqrt{\pi(t-\tau)}}\, d\tau -\frac{ h(0^+)}{\sqrt{\pi t}} \\
&= \int_0^\infty \Big(p_0(\xi) - \frac{p_0(\xi) + \sqrt{\sigma} p_0(-\sqrt{\sigma}\xi)}{1+\sqrt{\sigma}}\Big) \frac{\xi}{2\sqrt{\pi t^3}}e^{-\frac{\xi^2}{4t}} \,d\xi \\
&= \int_0^\infty \frac{\sqrt{\sigma}(p_0(\xi) - p_0(-\sqrt{\sigma}\xi))}{1+\sqrt{\sigma}} \frac{\xi}{2\sqrt{\pi t^3}}e^{-\frac{\xi^2}{4t}} \,d\xi.
\end{aligned}
\]
Hence for a fixed $t > 0$,
\[
\left| \frac{\partial p}{\partial y} (t,0^+; \epsilon) \right| \le C \frac{\sqrt{\sigma}}{1 + \sqrt{\sigma}} = O(\sqrt{\sigma}) = O(\epsilon^{0.5 - q}).
\]

(ii) Consider the case $q>0.5$.
Note that
\[
\sigma=\eps^{1-2q}\to \infty\quad\mbox{and}\quad\|p_0\|_{L^\infty(\R^-)}=\epsilon^q \|\phi\|_{L^\infty(\R^-)} \to0
\] 
as $\eps\to 0$. By \eqref{h(t)}, we have
\begin{align*}
|h(t; \epsilon)| &= \left|\int_0^\infty \frac{p_0(\xi)+\sqrt{\sigma}p_0(-\sqrt{\sigma }\xi)}{1+\sqrt{\sigma}}\frac{e^{-\frac{\xi^2}{4t}}}{\sqrt{\pi t}}\,d\xi\right|\\
&\le \frac{\|p_0\|_{L^\infty(\R^+)}+\sqrt{\sigma}\|p_0\|_{L^\infty(\R^-)}}{1+\sqrt{\sigma}}\\
&= \frac{\|p_0\|_{L^\infty(\R^+)} \epsilon^{q-0.5} + \epsilon^q \|\phi\|_{L^\infty(\R^-)}}{\epsilon^{q-0.5}+1}\\
&= O(\epsilon^{q-0.5}),
\end{align*}
as $\epsilon \to 0$.
Similarly, we can check that $h'(t; \epsilon) \to 0$ uniformly as $\epsilon \to 0$. Therefore by \eqref{pright},
\[
\lim_{\epsilon \to 0} p(t,y; \epsilon) = \int_0^\infty p_0(\xi) \Phi(t, y,\xi)\,d\xi
\]
for all $y > 0$. Hence the limit exists. The boundary behavior is already proved because $|p(t, 0^+; \epsilon)| = |h(t; \epsilon)| = O(\epsilon^{q-0.5})$.
\end{proof}

\section{Explicit solutions}\label{sect.examples}

For special initial values, the value $h(t)$ at the interface is given in a simple form, and hence the solution $p(t,y)$ and the limit as $\eps\to0$ are given in simpler forms. In this section, we present two cases of explicit solutions. These functions provide useful information and intuition about heterogeneous diffusion equation \eqref{1.3}.

\subsection{Fundamental solution}

We consider the explicit solution when the initial value is the delta distribution placed at $x_0>0$, i.e.,
\[
\phi(x)=\delta(x-x_0).
\]
Although $\phi\notin L^\infty$, we can construct the fundamental solution. First, we calculate the boundary value $h(t)$ using \eqref{h(t)}. We write the initial condition in terms of the potential  variable $p$, which is still
\[
p_0(y) = \delta(y-x_0).
\]
Putting this into \eqref{h(t)} gives
\[h(t)=  \frac{e^{-\frac{x_0^2}{4t}}}{(1+\sqrt{\sigma})\sqrt{\pi t}}.\]
Then, \eqref{pright} and \eqref{pleft} gives the explicit solution
\[
p(t,y) = \begin{cases}\Phi(t,y,x_0) + \frac{y}{2\sqrt{\pi}}\int_0^t \frac{1}{1+\sqrt{\sigma}}\frac{e^{-\frac{x_0^2}{4(t-\tau)}}}{\sqrt{\pi(t-\tau)}}\frac{e^{-\frac{y^2}{4\tau}}}{\sqrt{\tau^3}} \,d\tau,&y>0,\\
	- \frac{y}{2\sqrt{\pi}}\int_0^t \frac{1}{\sqrt{\sigma}(1+\sqrt{\sigma})}\frac{e^{-\frac{x_0^2}{4(t-\tau)}}}{\sqrt{\pi(t-\tau)}}\frac{e^{-\frac{y^2}{4\sigma\tau}}}{\sqrt{\tau^3}} \,d\tau,&y<0.
	
	\end{cases}
\]
After a simplification, we obtain the fundamental solution of \eqref{1.3},
\begin{equation}\label{fundamental}
p(t,y)=\begin{cases}\frac{1}{2 \sqrt{\pi t}} e^{-\frac{(y-x_0)^2}{4t}} + \frac{1-\sqrt{\sigma}}{2(1+\sqrt{\sigma})}\frac{1}{\sqrt{\pi t}}e^{-\frac{(y+x_0)^2}{4t}},&y>0,\\
	\frac{1}{(1+\sqrt{\sigma})\sqrt{\pi t}}e^{-\frac{(x_0-\frac{y}{\sqrt{\sigma}})^2}{4t}},&y<0.
	\end{cases}
\end{equation}

The limit of the solution as $\eps\to0$ is divided into three cases. If $q>0.5$, $\sigma\to\infty$ as $\eps\to0$. Then, $p(t,y)$ in \eqref{fundamental} converges to
\[
p(t,y)=\begin{cases}\frac{1}{2 \sqrt{\pi t}} e^{-\frac{(y-x_0)^2}{4t}} -\frac{1}{2\sqrt{\pi t}}e^{-\frac{(y+x_0)^2}{4t}},&y>0,\\
	0,&y<0.
	\end{cases}
\]
One can easily check that $p(t,0^+)=0$ and $p_y(t,0^+)\ne0$ in the case. If $q<0.5$, then $\sigma\to0$ as $\eps\to0$, and $p(t,y)$ in \eqref{fundamental} converges to
\[
p(t,y)=\begin{cases}\frac{1}{2 \sqrt{\pi t}} e^{-\frac{(y-x_0)^2}{4t}} + \frac{1}{2\sqrt{\pi t}}e^{-\frac{(y+x_0)^2}{4t}},&y>0,\\
	0,&y<0.
	\end{cases}
\]
One can easily check that $p(t,0^+)\ne0$ and $p_y(t,0^+)=0$ in the case. 

The third case $q=0.5$ is special. In this case, $\sigma=1$ for all $\eps>0$. Hence, the solution in \eqref{fundamental} is independent of $\eps$. In fact, the variable $y$ fully absorbed the effect of the heterogeneity. Then, $p(t,y)$ in \eqref{fundamental} is written as
\[
p(t,y)=\frac{1}{2\sqrt{\pi t}}e^{-\frac{(y-x_0)^2}{4t}},\quad y\in\R,
\]
which is the heat kernel on the whole real line $\R$. One can easily check that $p(t,0^+)\ne0$ and $p_y(t,0^+)\ne0$ in the case. Both Neumann and Dirichlet boundary conditions fail when $q=0.5$.

\subsection{Step function initial value}

We consider the explicit solution when the initial value is a step function,
\[\phi(x) = \begin{cases}
	a,&x>0,\\
	b,&x<0.
\end{cases}
\]
Then, the corresponding initial value for the potential $p$ is
\[
p_0(x) = \begin{cases}
	a,&x>0,\\
	\eps^{q}b,&x<0.
\end{cases}
\]
Putting this initial value into \eqref{h(t)} gives
\[
h(t) = \frac{a+\eps^{q}b}{1+\sqrt{\sigma}}.
\]
Then, \eqref{pright} gives the explicit solution
\begin{equation*}
		p(t,y) = \begin{cases}
			a(1-\erfc\big(\frac{y}{2\sqrt{t}}\big))+ \frac{a+\eps^{q}b}{1+\sqrt{\sigma}}\erfc\big(\frac{y}{2\sqrt{t}}\big),& y>0,\\
			\eps^q b(1-\erfc\big(\frac{-y}{2\sqrt{\sigma t}}\big))+ \frac{a+\eps^{q}b}{1+\sqrt{\sigma}}\erfc\big(\frac{-y}{2\sqrt{\sigma t}}\big),& y<0.
		\end{cases}
\end{equation*}
After a simplification, we obtain the solution of \eqref{1.3},
\begin{equation}\label{3.2}
	p(t,y) = \begin{cases}
		a+\frac{\eps^{q}b-\sqrt{\sigma}a}{1+\sqrt{\sigma}}\frac{2}{\sqrt{\pi}}\int_{\frac{y}{2\sqrt{t}}}^\infty e^{-\eta^2}\, d\eta,& y>0,\\
		\eps^q b+\frac{a-\sqrt{\sigma}\eps^{q}b}{1+\sqrt{\sigma}}\frac{2}{\sqrt{\pi}}\int_{\frac{-y}{2\sqrt{\sigma t}}}^\infty e^{-\eta^2}\, d\eta,& y<0.
	\end{cases}
\end{equation}

The limit of the solution as $\eps\to0$ is divided into three cases. If $q>0.5$, $\sigma\to\infty$ as $\eps\to0$. Then, $p(t,y)$ in \eqref{3.2} converges to
\[
p(t,y) = \begin{cases}
		a-a\frac{2}{\sqrt{\pi}}\int_{\frac{y}{2\sqrt{t}}}^\infty e^{-\eta^2}\, d\eta,& y>0,\\
		0& y<0.
	\end{cases}
\]
If $q<0.5$, $\sigma\to0$ as $\eps\to0$, and $p(t,y)$ in \eqref{3.2} converges to
\[
p(t,y) = \begin{cases}
		a,\quad& y>0,\\
		0,& y<0.
	\end{cases}
\]
If $q=0.5$, $\sigma=1$ for all $\eps>0$, and $p(t,y)$ in \eqref{3.2} converges to
\[
p(t,y) = \begin{cases}
		a-\frac{a}{\sqrt{\pi}}\int_{\frac{y}{2\sqrt{t}}}^\infty e^{-\eta^2}\, d\eta,& y>0,\\
		\frac{a}{\sqrt{\pi}}\int_{\frac{-y}{2\sqrt{t}}}^\infty e^{-\eta^2}\, d\eta,& y<0.
	\end{cases}
\]

\section{Numerical test}\label{sect.numerics}

In this section, we test the boundary condition \eqref{BC} and the decay rate \eqref{decay} numerically. Solving \eqref{1.3} directly is tricky due to the discontinuity of the solution $u$ at the  interface $x=0$. One way to handle the discontinuity is to write it in terms of pressure,
\[
p(t,x):=D^qu(t,x).
\]
Then, \eqref{1.3} is written as
\begin{equation}\label{4.1}
\begin{cases}
D^{-q}p_t = (D^{1-q} p_x)_x,& t>0,\ x\in\R,\\
p(0,x) = D^q\phi(x), & x\in\R.
\end{cases}
\end{equation}
We may solve \eqref{4.1} for $p(t,x)$ and then return to $u(t,x)$. Note that both equations, \eqref{2.1} and \eqref{4.1}, are for the same diffusion pressure $p=D^qu$, where the difference is in the space variables $x$ and $y$.

Four snapshots of solutions of the two cases, $u(t,x)$ and $p(t,y)$, are given in Figure \ref{fig1}, when
\[
\phi(x)=1,\quad \Delta x=0.002,\quad \text{and }\ \ q=0.9.
\]
We used a Matlab function `pdepe' for the computation. The snapshots are taken at the moment $t=0.01$ for four cases with $\eps=2^{-k}$ with $k=1,2,3,$ and $4$. According to the variable changes, $u=p$ in the region $x>0$, which can be observed in the figures. 

\begin{figure}[h!]
\centering
\includegraphics[width=0.36\textwidth]{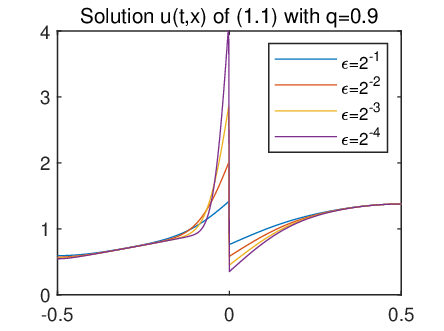}
\includegraphics[width=0.36\textwidth]{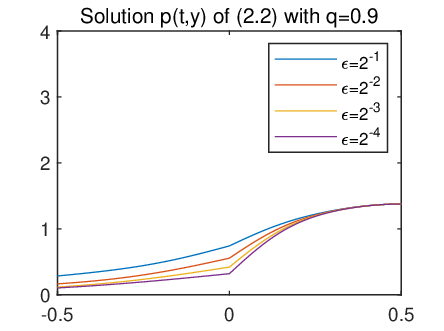}
\caption{Snapshots at $t=0.01$ are identical for $x,y>0$.}
\label{fig1}
\end{figure}

We observe that the boundary values $u(t,0^+)\to0$ as $\eps\to0$. However, we could not decrease $\eps$ further when we solve \eqref{4.1} numerically. It requires a smaller mesh size. On the other hand, solving \eqref{2.1} is handier as long as the domain for $y<0$ is taken properly. The boundary values $u(t,0^+)$ are displayed as $\eps\to0$ in the left of Figure \ref{fig2}. The boundary value $u(t,0^+)$ is computed for 50 cases of $\eps\in(10^{-4},10^0)$ in a log scale. We can observe that they make a line of slope $k=0.39$ for the case $q=0.9$. This implies that the relation between the boundary value and the parameter $\eps$ is approximately
\[
u(t,0^+)=O(\eps^{k})\ \ \text{with}\ \ k\cong0.39 \ \ \text{if}\ \ q=0.9.
\]
On the other hand, the graph for the case $q=0.6$ is slightly curved. If we take the parameter regime $10^{-4}<\eps<10^{-2}$, we have
\[
u(t,0^+)=O(\eps^{k})\ \ \text{with}\ \ k\cong0.082 \ \ \text{if}\ \ q=0.6.
\]
This shows that the boundary value goes to zero slowly when $q$ is close to $0.5$. In the third figure, the power $k$ is computed for 50 cases of $q\in(0.5,1)$. We can see that $k\to0$ as $q\to0.5$ almost linearly. The numerical observations agree with our decay estimate \eqref{decay}, which is $k \simeq q - 0.5$.

\begin{figure}[h!]
\centering
\includegraphics[width=0.32\textwidth]{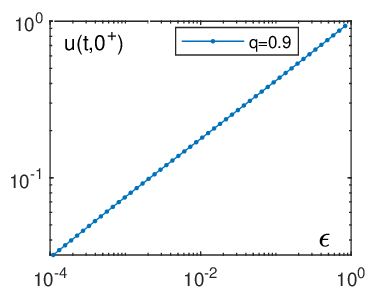}
\includegraphics[width=0.32\textwidth]{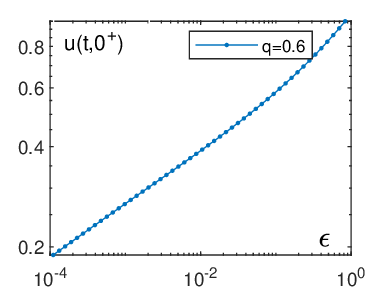}
\includegraphics[width=0.32\textwidth]{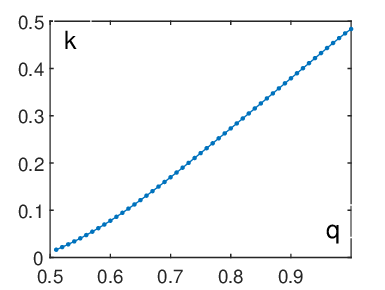}
\caption{Power laws between $u(t,0^+)$ and $\eps$ when $\eps\to0$.}
\label{fig2}
\end{figure}

Four snapshots of solutions of \eqref{2.1} are given in Figure \ref{fig3} when
\[
\phi(x)=1+\sin(x),\quad \Delta x=0.002,\quad \text{and }\ \ q=0.1.
\]
The snapshots are taken at the moment $t=0.1$ for four cases with $\eps=10^{-k}$ with $k=0,1,2,$ and $3$. The left figure is for the solution density $u(t,x)$. The boundary value increases as $\eps\to0$. The second figure is for the gradient of the solution $\partial_x u(t,x)$. We can observe that the derivative $\partial_x u(t,0^+)$ decreases to 0 as $\eps\to0$.

\begin{figure}[h!]
\centering
\includegraphics[width=0.36\textwidth]{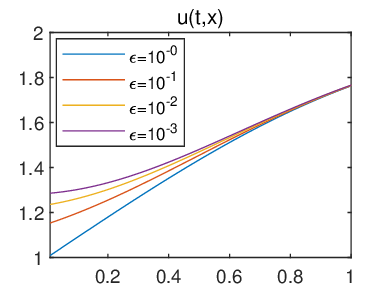}
\includegraphics[width=0.36\textwidth]{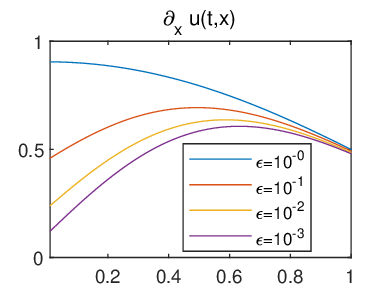}
\caption{Snapshots at $t=0.1$. Slope approaches 0 as $\eps\to0$.}
\label{fig3}
\end{figure}

In the left of Figure \ref{fig4}, the slopes of the solution at the boundary, $\partial_x u(t,0^+)$, are given as $\eps\to0$. The graph consists of 50 cases of $\eps\in(10^{-4},10^0)$ in a log scale. We can observe that they make almost a line of slope $k=0.34$ for the case $q=0.1$. This implies that the relation between the boundary slope and the parameter $\eps$ is
\[
\partial_x u(t,0^+)=O(\eps^{k})\ \ \text{with}\ \ k\cong0.34 \ \ \text{if}\ \ q=0.1.
\]
On the other hand, the graph for the case $q=0.4$ is curved. If we take the parameter regime $10^{-4}<\eps<10^{-3}$, we have
\[
\partial_x u(t,0^+)=O(\eps^{k})\ \ \text{with}\ \ k\cong0.063 \ \ \text{if}\ \ q=0.4.
\]
These numerical simulations show that the boundary slope goes to zero slower when $q$ approaches $0.5$. In the third figure, the power $k$ is computed for 50 cases of $q\in(0,0.5)$. We can see that $k\to0$ as $q\to0.5$ almost linearly. The numerical observations agree with our decay estimate \eqref{decay}, which suggests $k \simeq 0.5 - q$.

\begin{figure}[h!]
\centering
\includegraphics[width=0.32\textwidth]{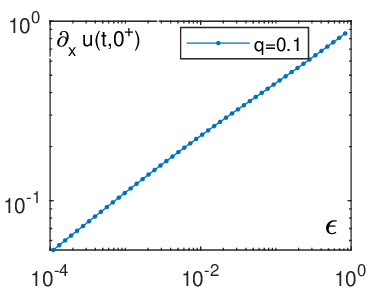}
\includegraphics[width=0.32\textwidth]{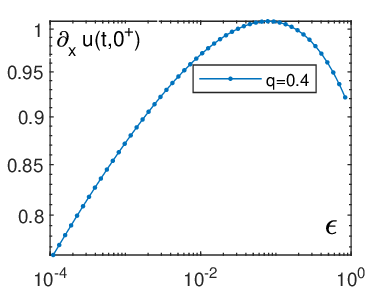}
\includegraphics[width=0.32\textwidth]{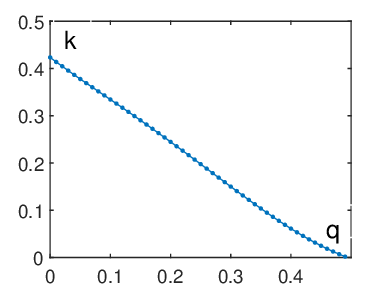}
\caption{Power laws between $\partial_x u(t,0^+)$ and $\eps$ when $\eps\to0$.}
\label{fig4}
\end{figure}

\section{Why is the boundary condition split at $q=0.5$?}\label{sect.why}

Theorem \ref{thm} says that the Neumann boundary condition emerges as $\eps\to0$ when $q<0.5$, and the Dirichlet boundary condition when $q>0.5$. The reason why the boundary condition splits along the case $q=0.5$ can be explained using a random walk model. In a one-dimensional position jump process, particles jump walk length $\tri x$ every sojourn time $\tri t$ toward a randomly chosen direction. If $\tri x$ and $\tri t$ are constant, the diffusivity is constant and given by
\[
D=\frac{(\tri x)^2}{2\tri t}.
\]
If $\tri x$ and $\tri t$ converge to zero keeping the ratio $D$ unchanged, the particle density $u(t,x)$ satisfies a diffusion equation
\begin{equation}\label{HeatE}
u_t=Du_{xx}.
\end{equation}
Equation \eqref{HeatE} is equivalent to \eqref{1.3} for any $q\in\R$ when $D$ is constant. 

To obtain the nonconstant diffusivity $D$ in \eqref{D}, we consider a reversible random walk system taking nonconstant $\tri x$ and $\tri t$. Diffusion equations for such spatially heterogeneous cases have been recently studied in \cite{009,007,008}. Let $\delta>0$ be a small constant, $\ell$ and $\tau$ be positive functions defined on $\R$, and
\[
\tri x\big|_{\text{at } x}=\ell(x)\delta\quad\text{and}\quad \tri t\big|_{\text{at } x}=\tau(x)\delta^2.
\]
Then, the corresponding heterogeneous diffusivity is
\[
D(x)=\frac{\ell^2(x)}{2\tau(x)},
\]
and the diffusion equation obtained after taking $\delta\to0$ limit is
\begin{equation}\label{5.2}
u_t=\frac{1}{2}\big(K\big(Mu\big)_x\big)_x,\quad K=\ell,\ M=\frac{\ell}{\tau}.
\end{equation}
If $\tau$ is constant, the diffusion equation \eqref{5.2} is equivalent to \eqref{1.3} with $q=0.5$. This is why the case $q=0.5$, Wereide's diffusion law, is the border case. The corresponding exponent $q\in\R$ depends on the heterogeneity in $\tri t$, and there is no universal one but different in each case.

There is an essential difference between the sojourn time $\tri t$ and the walk length $\tri x$. For example, a change of $\tri x$ in the region $x<0$ does not make any difference in the other region $x>0$. For example, we have changed the scale of the space variable using the relation \eqref{2.0}, which corresponds to the change of $\tri x$ in the region $x<0$. However, it  does not change the density in the region $x>0$ as we can see in Figure \ref{fig1}. On the other hand, if $\tri t$ changes in the $x<0$ region, the particle density in the $x>0$ region changes. 

Next, we consider the relation between $\tau$ and $q$. Let
\[
\ell(x)=\begin{cases}
\eta,&x\le0,\\
1,&x\ge0,
\end{cases}
\qquad
\tau(x)=\begin{cases}
\zeta,&x\le0,\\
0.5,&x\ge0.
\end{cases}
\]
Then, the diffusivity is given by
\[
D(x)=\begin{cases}
\frac{\eta^2}{2\zeta},&x<0,\\
1,&x>0.
\end{cases}
\]
To obtain the diffusivity in \eqref{D}, we choose $\zeta$ and $\eta$ so that $\frac{\eta^2}{2\zeta}=\eps$. If we take $\zeta=0.5$, the diffusion law \eqref{5.2} corresponds to the case $q=0.5$. For general $q\in\R$, relations among $\eps,\zeta,\eta,$ and $q$ are
\[
\eps^{1-q}=\eta\quad\text{and}\quad \eps^q=\frac{\eta}{\zeta}.
\]
Therefore, $\zeta=\eps^{1-2q}$. In other words, the parameter $\sigma$ is actually the sojourn time $\zeta$ in the $x<0$ region. If $q<0.5$, $\zeta\to0$ as $\eps\to0$. This implies that the particles in $x<0$ region jump so frequently as $\eps\to0$ and hence collide the interface $x=0$ infinitely times more frequently from the region $x<0$. This makes particles in $x>0$ cannot leave the region after taking the limit as $\eps\to0$, which gives the Neumann boundary condition. On the other hand, if $q>0.5$, $\zeta\to\infty$ as $\eps\to0$. In this case, particles do not hit the interface $x=0$ from the left side almost surely. Therefore, there is no chance for a particle that left the $x>0$ region to return to it. This gives the zero-Dirichlet boundary condition. Note that the diffusion law \eqref{5.2} is behind the match between $\zeta$ and $\sigma$, and the exact bisection of the divided boundary condition regimes across $q=0.5$. It is clear that there can be only one diffusion law that can provide such an exact match.

\subsection*{Acknowledgements}
J.C. was supported by the KERI Primary research program of MSIT/NST (No. 23A01002). Y.J.K was supported by National Research Foundation of Korea (NRF-2017R1A2B2010398).

\end{document}